\documentclass[12pt]{amsart}
\usepackage{amsfonts,amssymb,amscd,amsmath,enumerate,verbatim} 
\def\NZQ{\mathbb}               

\def\QQ{{\NZQ Q}}
\def\ZZ{{\NZQ Z}}
\def\RR{{\NZQ R}}

\def\frk{\mathfrak}               

\def\Phi{{\frk N}}
\def\eb{{\bold e}}

\def\opn#1#2{\def#1{\operatorname{#2}}} 
\opn\chara{char} 
\opn\length{\ell} 
\opn\pd{pd} 
\opn\rk{rk}
\opn\projdim{proj\,dim} 
\opn\injdim{inj\,dim} 
\opn\rank{rank}
\opn\depth{depth} 
\opn\grade{grade} 
\opn\height{height}
\opn\embdim{emb\,dim} 
\opn\codim{codim}

\opn\Tr{Tr} 
\opn\bigrank{big\,rank}
\opn\superheight{superheight}
\opn\lcm{lcm}
\opn\trdeg{tr\,deg}
\opn\reg{reg} 
\opn\lreg{lreg} 
\opn\ini{in} 
\opn\lpd{lpd}
\opn\size{size}
\opn\mult{mult}
\opn\dist{dist}
\opn\cone{cone}
\opn\lex{lex}
\opn\rev{rev}
\opn\div{div} \opn\Div{Div} \opn\cl{cl} \opn\Cl{Cl}
\opn\Spec{Spec} \opn\Supp{Supp} \opn\supp{supp} \opn\Sing{Sing}
\opn\Ass{Ass} \opn\Min{Min}
\opn\Ann{Ann} \opn\Rad{Rad} \opn\Soc{Soc}
\opn\Syz{Syz} \opn\Im{Im} \opn\Ker{Ker} \opn\Coker{Coker}
\opn\Am{Am} \opn\Hom{Hom} \opn\Tor{Tor} \opn\Ext{Ext}
\opn\End{End} \opn\Aut{Aut} \opn\id{id} \opn\ini{in}

\opn\nat{nat}
\opn\pff{pf}
\opn\Pf{Pf} \opn\GL{GL} \opn\SL{SL} \opn\mod{mod} \opn\ord{ord}
\opn\Gin{Gin}
\opn\Hilb{Hilb}\opn\adeg{adeg}\opn\std{std}\opn\ip{infpt}
\opn\Pol{Pol}
\opn\sat{sat}
\opn\Var{Var}
\opn\Gen{Gen}

\opn\aff{aff} \opn\con{conv} \opn\relint{relint} \opn\st{st}
\opn\lk{lk} \opn\cn{cn} \opn\core{core} \opn\vol{vol}
\opn\link{link} \opn\star{star}
\opn\gr{gr}

\def\Pc{{\mathcal P}}
\def\Qc{{\mathcal Q}}
\def\Rc{{\mathcal R}}

\def\pot#1#2{#1[\kern-0.28ex[#2]\kern-0.28ex]}

\opn\dirlim{\underrightarrow{\lim}}
\opn\inivlim{\underleftarrow{\lim}}
\def\Implies{\ifmmode\Longrightarrow \else
	\unskip${}\Longrightarrow{}$\ignorespaces\fi}
\def\implies{\ifmmode\Rightarrow \else
	\unskip${}\Rightarrow{}$\ignorespaces\fi}
\def\iff{\ifmmode\Longleftrightarrow \else
	\unskip${}\Longleftrightarrow{}$\ignorespaces\fi}

\let\:=\colon
\newtheorem{Theorem}{Theorem}[section]
\newtheorem{Lemma}[Theorem]{Lemma}

\newtheorem{Proposition}[Theorem]{Proposition}
\newtheorem{Remark}[Theorem]{Remark}

\newtheorem{Example}[Theorem]{Example}

\newtheorem{Question}[Theorem]{Question}
\let\epsilon\varepsilon
\let\phi=\varphi
\let\kappa=\varkappa
\def\qed{\ifhmode\textqed\fi
	\ifmmode\ifinner\quad\qedsymbol\else\dispqed\fi\fi}
\def\textqed{\unskip\nobreak\penalty50
	\hskip2em\hbox{}\nobreak\hfil\qedsymbol
	\parfillskip=0pt \finalhyphendemerits=0}
\def\dispqed{\rlap{\qquad\qedsymbol}}

\opn\dis{dis}
\opn\height{height}
\opn\dist{dist}
\def\pnt{{\raise0.5mm\hbox{\large\bf.}}}

\opn\Lex{Lex}

\begin{document}
\title{Ehrhart polynomials with negative coefficients}
\author{Takayuki Hibi, Akihiro Higashitani, Akiyoshi Tsuchiya 
\\ and Koutarou Yoshida}
\address{Takayuki Hibi,
Department of Pure and Applied Mathematics,
Graduate School of Information Science and Technology,pe
Osaka University,
Toyonaka, Osaka 560-0043, Japan}
\email{hibi@math.sci.osaka-u.ac.jp}
\address{Akihiro Higashitani,
Department of Mathematics, Kyoto Sangyo University, 
Motoyama, Kamigamo, Kita-Ku, Kyoto, Japan, 603-8555}
\email{ahigashi@cc.kyoto-su.ac.jp}
\address{Akiyoshi Tsuchiya,
Department of Pure and Applied Mathematics,
Graduate School of Information Science and Technology,
Osaka University,
Toyonaka, Osaka 560-0043, Japan}
\email{a-tsuchiya@cr.math.sci.osaka-u.ac.jp}
\address{Koutarou Yoshida,
Department of Pure and Applied Mathematics,
Graduate School of Information Science and Technology,
Osaka University,
Toyonaka, Osaka 560-0043, Japan}
\email{u912376b@ecs.osaka-u.ac.jp}

\begin{abstract}
It is shown that, for each $d \geq 4$, there exists
an integral convex polytope $\Pc$ 
of dimension $d$ such that each of the coefficients of 
$n, n^{2}, \ldots, n^{d-2}$ of its Ehrhart polynomial
$i(\Pc,n)$ is negative.  
Moreover, it is also shown that for each $d \geq 3$ and $1 \leq k \leq d-2$, 
there exists an integral convex polytope $\Pc$ of dimension $d$ such that 
the coefficient of $n^k$ of the Ehrhart polynomial $i(\Pc,n)$ of $\Pc$ is negative 
and all its remaining coefficients are positive. 
Finally, 
we consider all the possible sign patterns of the coefficients of the Ehrhart polynomials 
of low dimensional integral convex polytopes.
\end{abstract}
\maketitle
\section*{Introduction}
A convex polytope is called {\em integral} if any of its vertices has
integer coordinates.  
Let $\Pc \subset \RR^N$ be an integral convex polytope of dimension $d$. 
We define the function $i(\Pc,n)$ by setting 
\[
i(\Pc,n) = \sharp(n\Pc \cap \ZZ^N) \;\; \text{for} \;\; n = 1, 2, \ldots,
\] 
where $n\Pc = \{ \, n \alpha \, : \, \alpha \in \Pc \, \}$ and $\sharp X$ is the cardinality of a finite set $X$.
The study on $i(\Pc,n)$ originated in Ehrhart \cite{Ehrhart} who proved that 
$i(\Pc,n)$ is a polynomial in $n$ of degree $d$ with the constant term 1. 
Furthermore, the coefficients of $n^{d}$ and $n^{d-1}$ of $i(\Pc,n)$
are always positive (\cite[Corollary 3.20 and Theorem 5.6]{BeckRobins}).
We say that $i(\Pc,n)$ is the {\em Ehrhart polynomial} of $\Pc$. 

In his talk of the Clifford Lectures at Tulane University, 25--27 March 2010,
Richard Stanley gave an Ehrhart polynomial with
a negative coefficient.  More precisely,
the polynomial $\frac{\,13\,}{6}n^{3} + n^{2} - \frac{\,1\,}{6}n + 1$
is the Ehrhart polynomial of
the tetrahedron in $\RR^{3}$
with vertices $(0,0,0), (1,0,0), (0,1,0)$ and $(1,1,13)$.
See \cite[Example 3.22]{BeckRobins}.
His talk naturally inspired us to find
integral convex polytopes
of dimension $\geq 4$ whose Ehrhart polynomials possess negative coefficients.
Consult 
\cite[Part II]{HibiRedBook} and \cite[pp.~235--241]{StanleyEC}
for fundamental materials on Ehrhart polynomials.

The primary purpose of the present paper is, for each $d \geq 4$, 
to show the existence of an integral
convex polytope of dimension $d$
such that each of the coefficients of 
$n, n^{2}, \ldots, n^{d-2}$ of its Ehrhart polynomial
$i(\Pc,n)$ is negative.  In fact,
\begin{Theorem}\label{main}
	Given an arbitrary integer $d \geq 4$, 
	there exists an integral convex polytope $\Pc$ of dimension $d$ 
	such that 
	each of the coefficients of $n$, $n^2, \ldots, n^{d-2}$ of 
	the Ehrhart polynomial $i(\Pc,n)$ of $\Pc$ is negative. 
\end{Theorem}
This theorem says that all the possible coefficients of Ehrhart polynomials can be negative.

The second purpose of the present paper is the study of the sign patterns of 
the coefficients of the Ehrhart polynomials of integral convex polytopes. 
In the present paper, as a futher investigation on the sign of the coefficients of Ehrhart polynomials, 
we prove the following theorem. 
\begin{Theorem}\label{main1}
	Given arbitrary integers $d$ and $k$ with $1 \leq k \leq d-2$, 
	there exists an integral convex polytope $\Pc$ of dimension $d$ such that 
	the coefficient of $n^k$ of $i(\Pc,n)$ is negative and all its remaining coefficients are positive. 
\end{Theorem}

Finally, we also consider all the possible sign patterns of the coefficients of the Ehrhart polynomials 
of low dimensional integral convex polytopes (Proposition \ref{main2}).

\section{The existence of Ehrhart polynomials with negative coefficients}

In this section, we prove Theorem \ref{main}.
First, we collect some examples (Example \ref{ex0} and \ref{ex1}) 
which we will use in the proofs of Theorem \ref{main}, \ref{main1} and Proposition \ref{main2}. 

\begin{Example}\label{ex0}{\em 
		Let $m$ be an arbitrary positive integer and let 
		$$\ell_m=\{\alpha \in \RR : 0 \leq \alpha \leq m\}.$$ 
		Note that this $\ell_m$ is nothing but an integral convex polytope of dimension 1. 
		Then the Ehrhart polynomial $i(\ell_m,n)$ is equal to $mn+1$. 
	}\end{Example}
	
	\begin{Example}\label{ex1}{\em 
			It is known \cite[Example 3.22]{BeckRobins} that given an arbitrary integer $m \geq 1$, 
			there exists an integral convex polytope $\Qc_m^{(3)}$ of dimension 3 with 
			\[
			i(\Qc_m^{(3)}, n)=\frac{\,m\,}{6}n^3+n^2+\frac{\,-m+12\,}{6}n+1.
			\]
		}\end{Example}

	Next, we recall the following well-known fact. 
	\begin{Lemma}\label{lem}
		Let $f_1(n)$ and $f_2(n)$ be the Ehrhart polynomials of some integral convex polytopes of dimension $d_1$ and $d_2$, respectively. 
		Then there exists an integral convex polytope of dimension $d_1+d_2$ whose Ehrhart polynomial is equal to $f_1(n) \cdot f_2(n)$. 
	\end{Lemma}
	\begin{proof}
		Take the direct product of two integral convex polytopes. 
	\end{proof}


Our proof of Theorem \ref{main} will be given after preparing Lemma \ref{lem1}.
\begin{Lemma}
\label{lem1}
Let $d$ and $j$ be integers with $d \geq 5$ and $3 \leq j \leq d-2$, and
\[
g(d,j)=(d-3)^2 \binom{d-3}{j-1}-\binom{d-3}{j-3}.
\] 
Then one has $g(d,j) > 0$. 
\end{Lemma}
\begin{proof}
Since $d \geq 5$, one has  
$g(d,3)=(d-3)^2 \binom{d-3}{2}-1>0$ and
$g(d,d-2)=(d-3)^2-\binom{d-3}{2}>0$. 
Thus $g(d,j) > 0$
for $j = 3$ and $j = d - 2$.
Especially the assertion is true for $d = 5$ and $d = 6$. 

We now work with induction on $d$. 
Let $d \geq 7$ and $4 \leq j \leq d-3$.  Then
\begin{align*}
g(d,j)&=((d-4)^2+2d-7)\left(\binom{d-4}{j-1}+\binom{d-4}{j-2}\right)
-\left(\binom{d-4}{j-3}+\binom{d-4}{j-4}\right)\\
&=g(d-1,j)+g(d-1,j-1)+(2d-7)\binom{d-3}{j-1
	}. 
\end{align*}
It follows from the assumption of induction that
$g(d-1,j)+g(d-1,j-1)>0$.
Hence $g(d,j)>0$, as desired. 
\end{proof}

Now, we prove Theorem \ref{main}.
\begin{proof}[Proof of Theorem \ref{main}]

Given an arbitrary integer $d \geq 4$,
from Example \ref{ex0} and \ref{ex1}, and by applying Lemma \ref{lem} repeatedly,
there exists
an integral convex polytope $\Pc_m^{(d)}$ of dimension $d$ such that 
\begin{align*}
i(\Pc_m^{(d)},n) & =i(\ell_{d-3},n)^{d-3}i(\Qc_m^{(3
	)},n) \\
& = ((d-3)n+1)^{d-3}\left(\frac{\,m\,}{6}n^3+n^2+\frac{\,-m+12\,}{6}n+1\right). 
\end{align*}
Let $i(\Pc_m^{(d)},n) = \sum_{i=0}^d c_i^{(d,m)}n^i$ with each $c_i^{(d,m)} \in \QQ$. Then 
\[
c_1^{(d,m)}=\frac{\,-m+12\,}{6}+A_1, \quad\quad
c_2^{(d,m)}=1+\frac{\,-m+12\,}{6}A_1+A_2
\]
and
\[
c_j^{(d,m)}=\frac{\,m\,}{6}A_{j-3}+A_{j-2}+\frac{\,-m+12\,}{6}A_{j-1}+A_j,
\quad\quad
3 \leq j \leq d-2, 
\]
where 
\[
A_i=(d-3)^i\binom{d-3}{i}, \quad\quad 0 \leq i \leq d-2. 
\]

Now, since each $A_j$ is independent of $m$, 
it follows that each of $c_1^{(d,m)}$ and $c_2^{(d,m)}$ is
negative for $m$ sufficiently large.
Let $3 \leq j \leq d-2$.  One has 
\begin{eqnarray*}
c_j^{(d,m)}&=&-\frac{\,A_{j-1}-A_{j-3}\,}{6}m+(A_{j-2}+2A_{j-1}+A_j) \\
&=&-(d-3)^{j-3}\frac{\,g(d,j)\,}{6}m+(A_{j-2}+2A_{j-1}+A_j),
\end{eqnarray*}
where $g(d,j)$ is the same function as in Lemma \ref{lem1}.
Since $g(d,j)>0$, it follows that 
$c_j^{(d,m)}$ can be negative for $m$ sufficiently large. 
Hence, for $m$ sufficiently large, the integral convex polytope 
$\Pc_m^{(d)}$ of dimension $d$ enjoys the required property.
\end{proof}

We conclude this section with

\begin{Remark}
{\em
The polynomial
\begin{align*}
i(\Qc_m^{(3)}, n) &= \frac{\,m\,}{6}n^3+n^2+\frac{\,-m+12\,}{6}n+1 \\
&= \frac{\,1\,}{6}(n + 1)(m n^{2} + (6 - m) n + 6)
\end{align*}
has a real positive zero for $m$ sufficient large.  Hence
$i(\Pc_m^{(d)},n)$ has a real positive zero for $m$ sufficient large. 

Thus in particular, for $m$ sufficient large and for an arbitrary integral
convex polytope $\Qc$, the Ehrhart polynomial
$i(\Pc_m^{(d)} \times \Qc,n)$ of $\Pc_m^{(d)} \times \Qc$ 
also possesses a negative coefficient.
}
\end{Remark} 

We are grateful to Richard Stanley for his suggestion on real positive roots of
Ehrhart polynomials.

\section{Ehrhart polynomials having only one negative coefficient}

In this section we prove Theorem \ref{main1}.
Let $\eb_i^d$ be the $i$th unit coordinate vector of $\RR^d$ for $1 \leq i \leq d$ and let ${\bf 0}^d$ be the origin of $\RR^d$. 
First, we give the following example.
		\begin{Example}\label{ex2}{\em 
				(a) Let $$\Pc^{(4)}=\con(\{{\bf 0}^4, \eb_1^4, \eb_2^4, \eb_3^4, \eb_1^4+26\eb_3^4+27\eb_4^4\}) \subset \RR^4.$$ 
				Then we have 
				$$i(\Pc^{(4)},n)=\frac{9}{8}n^4+\frac{31}{12}n^3+\frac{3}{8}n^2-\frac{1}{12}n+1.$$
				(b) Let \begin{align*}
					\Pc^{(5)}=\con(\{{\bf 0}^5, \eb_1^5, \eb_2^5, \eb_3^5, \eb_4^5, \eb_4^5+\eb_5^5, \eb_1^5+50\eb_4^5+51\eb_5^5\}) \subset \RR^5.
				\end{align*}
				Then we have 
				$$i(\Pc^{(5)},n)=\frac{13}{30}n^5+\frac{55}{24}n^4+\frac{37}{12}n^3+\frac{5}{24}n^2-\frac{1}{60}n+1.$$
			}\end{Example}

			\bigskip
			
			We also prepare the following two lemmas (Lemma \ref{lem2} and Lemma \ref{lem3}). 
			\begin{Lemma}\label{lem2}
				Given an arbitrary integer $d \geq 3$, there exists an integral convex polytope $\Pc$ of dimension $d$ 
				such that the coefficient of $n$ of $i(\Pc,n)$ is negative and all its remaining coefficients are positive. 
			\end{Lemma}
			\begin{proof}
				From Example \ref{ex1} and Example \ref{ex2}, we have
				\begin{align*}
					&i(\Qc_{13}^{(3)},n)=\frac{13}{6}n^3+n^2-\frac{1}{6}n+1; \\
					&i(\Pc^{(4)},n)=\frac{9}{8}n^4+\frac{31}{12}n^3+\frac{3}{8}n^2-\frac{1}{12}n+1; \\
					&i(\Pc^{(5)},n)=\frac{13}{30}n^5+\frac{55}{24}n^4+\frac{37}{12}n^3+\frac{5}{24}n^2-\frac{1}{60}n+1. 
				\end{align*}
				These examples show the required assertion in the cases $d=3,4$ and 5. 
				
				Let $f(n)=i(\Qc_{13}^{(3)},n)$, $g(n)=i(\Pc^{(4)},n)$, $h(n)=i(\Pc^{(5)},n)$ and 
				$$p(n)=i(\Qc_{12}^{(3)},n)=2n^3+n^2+1.$$ 
				By applying Lemma \ref{lem1} repeatedly, one sees that for each integer $s \geq 1$, 
				the polynomial $f(n)\cdot p(n)^s$ is the Ehrhart polynomial of an integral convex polytope of dimension $3(s+1)$. 
				Similarly, $g(n)\cdot p(n)^s$ and $h(n)\cdot p(n)^s$ are also the Ehrhart polynomials of 
				some integral convex polytopes of dimension $3s+4$ and dimension $3s+5$, respectively. 
				Thus, it is enough to prove that the coefficient of $n$ of each of the polynomials 
				$f(n) \cdot p(n)^s$, $g(n) \cdot p(n)^s$ and $h(n) \cdot p(n)^s$ is negative and all the remaining coefficients are positive. 
				
				We will prove by induction on $s$ that for each $s \geq 0$, the coefficient of $n$ of $f(n) \cdot p(n)^s$ 
				is equal to $-\frac{1}{6}$, the coefficient of $n^2$ is equal to $s+1$ 
				and all the remaining coefficients are positive. 
				Suppose that $s \geq 1$ and the polynomial $f(n) \cdot p(n)^{s-1}$ looks like 
				$$f(n) \cdot p(n)^{s-1}=a_{3s}n^{3s} + \cdots + a_3n^3+sn^2-\frac{1}{6}n+1,$$
				where each $a_j>0$. Then the direct computation shows that 
				the coefficients of $n, n^2, n^3$ and $n^4$ of $f(n) \cdot p(n)^s=(f(n) \cdot p(n)^{s-1})\cdot p(n)$ are as follows: 
				\begin{align*}
					&\text{(the coefficient of $n$ of $f(n) \cdot p(n)^s$)}=-\frac{1}{6}; \\
					&\text{(the coefficient of $n^2$ of $f(n) \cdot p(n)^s$)}=s+1; \\
					&\text{(the coefficient of $n^3$ of $f(n) \cdot p(n)^s$)}=a_3-\frac{1}{6}+2 > 0; \\
					&\text{(the coefficient of $n^4$ of $f(n) \cdot p(n)^s$)}=a_4+s-\frac{1}{3} > 0. 
				\end{align*}
				Moreover, since each $a_j$ is positive, all the coefficients of $n^r$ with $r \geq 5$ are also positive. 
				Hence, the coefficient of $n$ of $f(n) \cdot p(n)^s$ is equal to $-\frac{1}{6}$, 
				the coefficient of $n^2$ is equal to $s+1$ and all the remaining coefficients are positive. 
				Thus, by induction on $s$, we conclude that 
				the Ehrhart polynomial $f(n) \cdot p(n)^s$ satisfies the required condition for each $s \geq 0$. 
				In particular, the coefficient of $n$ of $f(n) \cdot p(n)^s$ is negative and all its remaining coefficients are positive. 
				
				By the same discussions as above, we can conclude that 
				the Ehrhart polynomials $g(n) \cdot p(n)^s$ and $h(n) \cdot p(n)^s$ also enjoy the required properties for each $s \geq 1$, 
				as desired. 
			\end{proof}
			
			\begin{Lemma}\label{lem3}
				Given an integer $e \geq 3$, 
				suppose that there is an integral convex polytope $\Pc$ of dimension $e$ such that 
				the coefficient of $n^r$ of $i(\Pc,n)$ is negative and all its remaining coefficients are positive for some $1 \leq r \leq e-2$. 
				Then there exists an integral convex polytope $\Pc'$ of dimension $e+1$ such that 
				the coefficient of $n^{r+1}$ of $i(\Pc',n)$ is negative and all its remaining coefficients are positive. 
			\end{Lemma}
			\begin{proof}
				Let $$i(\Pc,n)=a_en^e+\cdots+a_{r+1}n^{r+1}-a_rn^r+a_{r-1}n^{r-1}+\cdots+a_1n+1,$$ where each $a_j>0$. 
				By applying Lemma \ref{lem1}, we see that 
				there exists an integral convex polytope of dimension $e+1$ such that $i(\Pc',n)=(mn+1)i(\Pc,n)$. (See also Example \ref{ex0}.) Then 
				\begin{align*}
					i(\Pc',n)&=ma_en^{e+1} + (ma_{e-1}+a_e)n^e + \cdots +(ma_{r+1}+a_{r+2})n^{r+2} \\
					&+ (-ma_r+a_{r+1})n^{r+1}+(ma_{r-1}-a_r)n^r \\
					&+(ma_{r-2} + a_{r-1})n^{r-1} + \cdots +(m+a_1)n + 1. 
				\end{align*}
				Hence, for a sufficiently large integer $m$, the coefficient of $n^{r+1}$ of the Ehrhart polynomial $i(\Pc',n)$ 
				is negative and all its remaining coefficients are positive. 
			\end{proof}
			
			Now, we are in the position to give a proof of Theorem \ref{main1}. 
			
			\begin{proof}[Proof of Theorem \ref{main1}]
				For a given integer $d \geq 3$, Lemma \ref{lem2} directly proves the case where $k=1$. 
				Assume $k>1$ and $d > 3$. 
				
				From $k \leq d-2$, we have $d-k+1 \geq 3$. 
				Let $\Rc^{(d-k+1)}$ be an integral convex polytope of dimension $d-k+1$ 
				such that the coefficient of $n$ of $i(\Rc^{(d-k+1)},n)$ is negative and all its remaining coefficients are positive. 
				Note that the existence of such polytope is guaranteed by Lemma \ref{lem2}. 
				Applying Lemma \ref{lem3} for $\Rc^{(d-k+1)}$ repeatedly by $(k-1)$ times 
				proves the existence of an integral convex polytope $\Pc'$ of dimension $d(=(d-k+1)+(k-1))$ 
				satisfying that the coefficient of $n^k$ of $i(\Pc',n)$ is negative and all its remaining coefficients are positive, as desired. 
			\end{proof}

			\bigskip

			\section{A question on possible sign patters of coefficients}
			
			In this section, we consider the following question. 
			\begin{Question}\label{q}
				Given a positive integer $d \geq 3$ and integers $i_1,\ldots,i_q$ with $1 \leq i_1 < \cdots < i_q \leq d-2$, 
				does there exist an integral convex polytope of dimension $d$ whose Ehrhart polynomial satisfies 
				\begin{itemize}
					\item all the coefficients of $n^{i_1},\ldots,n^{i_q}$ are negative and 
					\item all the remaining coefficients are positive? 
				\end{itemize}
			\end{Question}

			We give a partial answer for Question \ref{q}. More precisely, we solve this in the case $d \leq 6$. 
			\begin{Proposition}\label{main2}
				Given an integer $3 \leq d \leq 6$ and integers $i_1,\ldots,i_q$ with $1 \leq i_1 < \cdots < i_q \leq d-2$, 
				there exists an integral convex polytope of dimension $d$ whose Ehrhart polynomial satisfies that 
				all the coefficients of $n^{i_1},\ldots,n^{i_q}$ are negative and all the remaining coefficients are positive. 
			\end{Proposition}
			\begin{proof}
				Theorem \ref{main1} guarantees the case $q=1$ for every $d \geq 3$. 
				Moreover, Theorem \ref{main} guarantees the case $q=d-2$ for every $d \geq 3$. 
				
				Thus, the remaining cases are as follows: 
				\begin{itemize}
					\item[(a)] $d=5, q=2$ and $(i_1,i_2)=(1,2),(1,3),(2,3)$; 
					\item[(b)] $d=6, q=2$ and $(i_1,i_2)=(1,2),(1,3),(1,4),(2,3),(2,4),(3,4)$; 
					\item[(c)] $d=6, q=3$ and $(i_1,i_2,i_3)=(1,2,3),(1,2,4),(1,3,4),(2,3,4)$. 
				\end{itemize}
				Each of these cases is guaranteed by Example \ref{Tsuchiya} below, as required. 
			\end{proof}
			
			\begin{Example}\label{Tsuchiya}{\em 
					(a) Let 
					\begin{align*}
						&\Pc_{1,2}^{(5)}=\con(\{{\bf 0}^5, \eb_1^5,\eb_2^5,\eb_3^5,\eb_4^5,\eb_1^5+\eb_2^5+99\eb_4^5+100\eb_5^5\}) \subset \RR^5, \\
						&\Pc_{1,3}^{(5)}=\con(\{{\bf 0}^5, \eb_1^5,\eb_2^5,\eb_3^5,\eb_4^5,3\eb_1^5+4\eb_2^5+5\eb_3^5+8\eb_4^5+371\eb_5^5\}) \subset \RR^5, \\
						&\Pc_{2,3}^{(5)}=\ell_{10} \times \ell_{10} \times \Qc_{100}^{(3)} \subset \RR^5, 
					\end{align*}
					where $\ell_{10}$ and $\Qc_{100}^{(3)}$ are the polytopes appearing in Example \ref{ex0} and Example \ref{ex1}, respectively. Then 
					\begin{align*}
						&i(\Pc_{1,2}^{(5)},n)=\frac{5}{6}n^5+\frac{17}{4}n^4+\frac{29}{6}n^3-\frac{9}{4}n^2-\frac{8}{3}n+1, \\
						&i(\Pc_{1,3}^{(5)},n)=\frac{371}{120}n^5+\frac{1}{8}n^4-\frac{1}{24}n^3+\frac{15}{8}n^2-\frac{1}{20}n+1, \\
						&i(\Pc_{2,3}^{(5)},n)=\frac{5000}{3}n^5+\frac{1300}{3}n^4-1430n^3-\frac{577}{3}n^2+\frac{16}{3}n+1. 
					\end{align*}
					
					(b) Let 
					\begin{align*}
						&\Pc_{1,2}^{(6)}=\con(\{{\bf 0}^6, \eb_1^6,\eb_2^6,\eb_3^6,\eb_4^6,\eb_5^6,\eb_1^6+\eb_2^6+999\eb_5^6+1000\eb_6^6\}) \subset \RR^6, \\
						&\Pc_{1,3}^{(6)}=\con(\{\eb_1^3,\eb_1^3+\eb_2^3, \eb_1^3+\eb_3^3,\eb_2^3+\eb_3^3,3\eb_1^3+4\eb_3^3,4\eb_1^3+\eb_2^3+3\eb_3^3\}) 
						\times \Qc_{26}^{(3)} \subset \RR^6, \\
						&\Pc_{1,4}^{(6)}=\Qc_{12}^{(3)} \times \Qc_{16}^{(3)} \subset \RR^6, \\
						&\Pc_{2,3}^{(6)}=\ell_2 \times \ell_2 \times \ell_2 \times \Qc_{30}^{(3)} \subset \RR^6, \\
						&\Pc_{2,4}^{(6)}=\ell_{40} \times \Pc_{1,3}^{(5)} \subset \RR^6, \\
						&\Pc_{3,4}^{(6)}=\con(\{{\bf 0}^6, \eb_1^6,\eb_2^6,\eb_3^6,\eb_4^6,\eb_5^6,\eb_1^6+\eb_2^6+\eb_3^6+999(\eb_4^6+\eb_5^6)+1000\eb_6^6\}) \subset \RR^6, 
					\end{align*}
					where $\Pc_{1,3}^{(5)}$ is the polytope appearing in (a) above. Then 
					\begin{align*}
						&i(\Pc_{1,2}^{(6)},n)=\frac{25}{18}n^6+\frac{751}{60}n^5+\frac{2515}{72}n^4+\frac{131}{6}n^3-\frac{2435}{72}n^2-\frac{617}{20}n+1, \\
						&i(\Pc_{1,3}^{(6)},n)=\frac{130}{9}n^6+\frac{137}{6}n^5+\frac{55}{9}n^4-\frac{2}{3}n^3+\frac{4}{9}n^2-\frac{1}{6}n+1, \\
						&i(\Pc_{1,4}^{(6)},n)=\frac{16}{3}n^6+\frac{14}{3}n^5-\frac{1}{3}n^4+4n^3+2n^2-\frac{2}{3}n+1, \\
						&i(\Pc_{2,3}^{(6)},n)=40n^6+68n^5+18n^4-17n^3-5n^2+3n+1, \\
						&i(\Pc_{2,4}^{(6)},n)=\frac{371}{3}n^6+\frac{971}{120}n^5-\frac{37}{24}n^4+\frac{1799}{24}n^3-\frac{1}{8}n^2+\frac{799}{20}n+1, \\
						&i(\Pc_{3,4}^{(6)},n)=\frac{25}{18}n^6+\frac{503}{120}n^5-\frac{241}{36}n^4-\frac{475}{24}n^3+\frac{281}{36}n^2+\frac{191}{10}n+1. 
					\end{align*}
					
					(c) Let 
					\begin{align*}
						&\Pc_{1,2,3}^{(6)}=\ell_1 \times \ell_1 \times \ell_1 \times \Qc_{40}^{(3)} \subset \RR^6, \\
						&\Pc_{1,2,4}^{(6)}=\Qc_{10}^{(3)} \times \Qc_{100}^{(3)} \subset \RR^6, \\
						&\Pc_{1,3,4}^{(6)}=\ell_1 \times \con(\{{\bf 0}^5, \eb_1^5,\eb_2^5,\eb_3^5,\eb_4^5,\eb_1^5+\eb_2^5+2\eb_3^5+10\eb_4^5+1000\eb_5^5\}) \subset \RR^6, \\
						&\Pc_{2,3,4}^{(6)}=\ell_3 \times \ell_3 \times \ell_3 \times \Qc_{40}^{(3)} \subset \RR^6. 
					\end{align*}Then \begin{align*}
					&i(\Pc_{1,2,3}^{(6)},n)=\frac{20}{3}n^6+21n^5+\frac{55}{3}n^4-\frac{10}{3}n^3-10n^2-\frac{5}{3}n+1, \\
					&i(\Pc_{1,2,4}^{(6)},n)=\frac{250}{9}n^6+\frac{55}{3}n^5-\frac{161}{9}n^4+4n^3-\frac{26}{9}n^2-\frac{43}{3}n+1, \\
					&i(\Pc_{1,3,4}^{(6)},n)=\frac{25}{3}n^6+\frac{26}{3}n^5-\frac{11}{3}n^4-\frac{7}{3}n^3+\frac{1}{3}n^2-\frac{1}{3}n+1, \\
					&i(\Pc_{2,3,4}^{(6)},n)=180n^6+207n^5-39n^4-\frac{250}{3}n^3-14n^2+\frac{13}{3}n+1. 
				\end{align*}

			}\end{Example}

\end{document}